\documentclass[12pt,reqno]{amsart}
\usepackage{amscd,amsmath,amsthm,amssymb}
\usepackage{color}
\usepackage{stmaryrd}
\usepackage{tikz}
\usepackage{url}

\usepackage{latexsym}
\usepackage{amsfonts,amsmath,mathtools}
\usepackage{graphics}
\usepackage{float}
\usepackage{enumitem}

\definecolor{verylight}{gray}{0.97}
\definecolor{light}{gray}{0.9}
\definecolor{medium}{gray}{0.85}
\definecolor{dark}{gray}{0.6}

\def\NZQ{\mathbb}               % the font for N,Z,Q,R,C

\def\ZZ{{\NZQ Z}}

\def\G{{\mathcal G}}

\def\pd{\textup{proj}\phantom{.}\!\textup{dim}}

\def\opn#1#2{\def#1{\operatorname{#2}}} % to make operators
\opn\chara{char} \opn\length{\ell} \opn\pd{pd} \opn\rk{rk}
\opn\projdim{proj\,dim} \opn\injdim{inj\,dim} \opn\rank{rank}
\opn\depth{depth} \opn\grade{grade} \opn\height{height}
\opn\embdim{emb\,dim} \opn\codim{codim}

\opn\Tr{Tr} \opn\bigrank{big\,rank}
\opn\superheight{superheight}\opn\lcm{lcm}
\opn\trdeg{tr\,deg}%\emph{
\opn\reg{reg} \opn\lreg{lreg} \opn\ini{in} \opn\lpd{lpd}
\opn\size{size} \opn\sdepth{sdepth}
\opn\link{link}\opn\fdepth{fdepth}\opn\lex{lex}
\opn\tr{tr}
\opn\type{type}
\opn\gap{gap}
\opn\diam{diam}
\opn\Mod{Mod}
\opn\div{div} \opn\Div{Div} \opn\cl{cl} \opn\Cl{Cl}
\opn\Spec{Spec} \opn\Supp{Supp} \opn\supp{supp} \opn\Sing{Sing}
\opn\Ass{Ass} \opn\Min{Min}\opn\Mon{Mon}
\opn\Ann{Ann} \opn\Rad{Rad} \opn\Soc{Soc}
\opn\Im{Im} \opn\Ker{Ker} \opn\Coker{Coker} \opn\Am{Am}
\opn\Hom{Hom} \opn\Tor{Tor} \opn\Ext{Ext} \opn\End{End}
\opn\Aut{Aut} \opn\id{id}

\opn\nat{nat}
\opn\pff{pf}%   \pf exists already
\opn\Pf{Pf} \opn\GL{GL} \opn\SL{SL} \opn\mod{mod} \opn\ord{ord}
\opn\Gin{Gin} \opn\Hilb{Hilb}\opn\sort{sort}
\opn\PF{PF}\opn\Ap{Ap}
\opn\dist{dist}
\opn\aff{aff}
\opn\relint{relint} \opn\st{st}
\opn\lk{lk} \opn\cn{cn} \opn\core{core} \opn\vol{vol}  \opn\inp{inp} \opn\nilpot{nilpot}
\opn\link{link} \opn\star{star}\opn\lex{lex}\opn\set{set}
\opn\width{wd}
\opn\Fr{F}
\opn\QF{QF}
\opn\G{G}
\opn\type{type}\opn\res{res}
\opn\conv{conv}
\opn\sr{sr}
\opn\gr{gr}

\def\pot#1#2{#1[\kern-0.28ex[#2]\kern-0.28ex]}

\opn\dirlim{\underrightarrow{\lim}}
\opn\inivlim{\underleftarrow{\lim}}
\def\Implies{\ifmmode\Longrightarrow \else
	\unskip${}\Longrightarrow{}$\ignorespaces\fi}
\def\implies{\ifmmode\Rightarrow \else
	\unskip${}\Rightarrow{}$\ignorespaces\fi}
\def\iff{\ifmmode\Longleftrightarrow \else
	\unskip${}\Longleftrightarrow{}$\ignorespaces\fi}

\let\:=\colon
\def\G{\mathcal{G}}
\newtheorem{theorem}{Theorem}[section]
\newtheorem{lemma}[theorem]{Lemma}
\newtheorem{corollary}[theorem]{Corollary}
\newtheorem{proposition}[theorem]{Proposition}
\newtheorem{question}[theorem]{Question}

\theoremstyle{definition}
\newtheorem{definition}[theorem]{Definition}
\newtheorem{setup}[theorem]{Setup}
\newtheorem{example}[theorem]{Example}

\theoremstyle{remark}
\newtheorem{remark}[theorem]{Remark}

\numberwithin{equation}{section}
\let\epsilon\varepsilon
\let\kappa=\varkappa
\def\qed{\ifhmode\textqed\fi
	\ifmmode\ifinner\hfill\quad\qedsymbol\else\dispqed\fi\fi}
\def\textqed{\unskip\nobreak\penalty50
	\hskip2em\hbox{}\nobreak\hfill\qedsymbol
	\parfillskip=0pt \finalhyphendemerits=0}
\def\dispqed{\rlap{\qquad\qedsymbol}}

\opn\dis{dis}
\def\pnt{{\raise0.5mm\hbox{\large\bf.}}}

\opn\Lex{Lex}

\def\aim{\textup{aim}}
\opn{\indm}{indm}

\begin{document}
	
\title{Matchings, Squarefree Powers and Betti Splittings}


%    author one information
\author{Marilena Crupi}
\address{Marilena Crupi, Department of mathematics and computer sciences, physics and earth sciences, University of Messina, Viale Ferdinando Stagno d'Alcontres 31, 98166 Messina, Italy}
\email{mcrupi@unime.it}

\author{Antonino Ficarra}
\address{Antonino Ficarra, Department of mathematics and computer sciences, physics and earth sciences, University of Messina, Viale Ferdinando Stagno d'Alcontres 31, 98166 Messina, Italy}
\email{antficarra@unime.it}

\author{Ernesto Lax}
\address{Ernesto Lax, Department of mathematics and computer sciences, physics and earth sciences, University of Messina, Viale Ferdinando Stagno d'Alcontres 31, 98166 Messina, Italy}
\email{erlax@unime.it}

%    \subjclass is required.

\date{}

\dedicatory{}

%    "Communicated by" -- provide editor's name; required.
\commby{}

\date{}

%    Abstract is required.
\begin{abstract}
	Let $G$ be a finite simple graph and let $I(G)$ be its edge ideal. In this article, we investigate the squarefree powers of $I(G)$ by means of Betti splittings. When $G$ is a forest, it is shown that the normalized depth function of $I(G)$ is non-increasing. Moreover, we compute explicitly the regularity function of squarefree powers of $I(G)$ with $G$ a forest, confirming a conjecture of Erey and Hibi.
\end{abstract}

\maketitle
	
\section*{Introduction}	
All the graphs we consider in this article are finite simple graphs. Let $G$ be a graph with the vertex set $V(G) = \{1, \ldots, n\}$ and the edge set $E(G)$. A $k$-\textit{matching} of $G$ is a subset $M$ of $E(G)$ of size $k$ such that $e\cap e'=\emptyset$ for all $e,e'\in M$ with $e\ne e'$. We denote by $V(M)$ the vertex set of $M$, that is, the set $\{i\in V(G):i\in e\ \text{for}\ e\in M\}$. We say that a graph $H$ is a \textit{subgraph of $G$} if $V(H)\subseteq V(G)$ and $E(H)\subseteq E(G)$. A subgraph $H$ of $G$ is said an \textit{induced subgraph} if for any two vertices $i, j$ in $H$, $\{i, j\}\in E(H)$ if and only if $\{i, j\}\in E(G)$. If $A$ is a subset of $V(G)$, the \textit{induced subgraph} on $A$, denoted by $G_A$, is the graph with vertex set $A$ and the edge set $\{\{i, j\}: \mbox{$i, j \in A$ and $\{i, j\}\in E(G)$}\}$. A matching $M$ is called an \textit{induced matching} if $E(G_{V(M)})=M$. The \textit{matching number} of $G$, denoted by $\nu(G)$, is the maximum size of a matching of $G$. Whereas, the \textit{induced matching number} of $G$, denoted by $\indm(G)$, is the maximum size of an induced matching of $G$. It is clear that $\indm(G)\le\nu(G)$ for any graph $G$.
	
Matchings play a pivotal role in graph theory \cite{H72}. Hereafter, we set $[n]=\{1,\dots,n\}$ if $n\ge1$ is an integer, and $[0]=\emptyset$. Let $G$ be a graph on the vertex set $[n]$ and let $S=K[x_1,\dots,x_n]$ be the standard graded polynomial ring over a field $K$. For a non--empty subset $A$ of $[n]$, we set ${\bf x}_A=\prod_{i\in A}x_i$.  Let $1\le k\le\nu(G)$. We denote by $I(G)^{[k]}$ the squarefree monomial ideal generated by ${\bf x}_{V(M)}$ for all $k$-matchings $M$ of $G$. If $k=1$, then $I(G)^{[1]}$ is the well--known ideal, called as the \textit{edge ideal} of $G$ \cite{RV, RVbook}, and we denote it simply by $I(G)$. More recently, matching theory has been related to Commutative Algebra \emph{via} the notion of \textit{squarefree powers} \cite{BHZN18,EHHM2022a}.
	
Let us explain the connection with the concept of squarefree power (see, for instance, \cite{BHZN18}). Let $I\subset S$ be a squarefree monomial ideal and $\G(I)$ be its unique minimal set of monomial generators. The \textit{$k$th squarefree power} of $I$, denoted by $I^{[k]}$, is the ideal generated by the squarefree monomials of $I^k$. Thus $u_1u_2\cdots u_k$, $u_i\in \G(I)$, $i\in[k]$, belongs to $\G(I^{[k]})$ if and only if $u_1,u_2,\dots,u_k$ is a regular sequence. Let $\nu(I)$ be the \textit{monomial grade} of $I$, \emph{i.e.}, the maximum among the lengths of a monomial regular sequence contained in $I$. Then $I^{[k]}\ne(0)$ if and only if $k\le\nu(I)$. Hence, the ideal $I(G)^{[k]}$ is the $k$th squarefree power of $I(G)$ and $\nu(I(G))=\nu(G)$.
	
The study of ordinary powers of ideals is a classical subject in Commutative Algebra. The fascination with this topic is due to the fact that many algebraic invariants behave asymptotically well, that is, stabilize or show a regular behaviour for sufficiently high powers. The first result in this area was obtained by Brodmann \cite{B79} who showed that the depth of the powers of an ideal $I$ of a Noetherian ring $R$ is eventually constant: there exists $k_0>0$ such that $\depth(R/I^{k})=\depth(R/I^{k_0})$ for all $k\ge k_0$. The study of the initial behaviour of the depth function of powers of graded ideals $I\subset S$ was initiated by Herzog and Hibi \cite{HH2005}. They made the boldest guess possible: any bounded convergent function $\varphi:\mathbb{Z}_{\ge0}\rightarrow\mathbb{Z}_{\ge0}$ is the depth function of some suitable graded ideal in a polynomial ring. Recently, this conjecture has been settled in affirmative by H\`a, Nguyen, Trung and Trung \cite[Theorem 4.1]{HNTT2021}.
	
Let $I\subset S$ be a squarefree monomial ideal and $k\in[\nu(I)]$. Let $d_k=\textup{indeg}(I^{[k]})$ be the \textit{initial degree} of $I^{[k]}$, that is, the minimum degree of a monomial of $\G(I^{[k]})$. In \cite[Proposition 1.1]{EHHM2022b}, the authors established that $$\depth(S/I^{[k]})\ge d_k-1$$ for all $k\in[\nu(I)]$. Based on this, they introduced the \textit{normalized depth function} of $I$, as follows
\[
g_I(k)=\depth(S/I^{[k]})-(d_k-1)
\]
where $k\in[\nu(I)]$. In contrast to the behaviour of the depth function of ordinary powers, it was predicted in \cite{EHHM2022b} that: \medskip
\\
\textbf{Conjecture:} \textit{For any $I\subset S$, the function $g_I(k)$ is a non-increasing function.}\medskip

Squarefree powers have been introduced in \cite{BHZN18}, and are currently studied by many researchers \cite{EF,EH2021,EHHM2022a,EHHM2022b,FPack2,FHH2022,SASF2022,SASF2023}.

Among graphs,  \textit{forests} are the simplest ones. In this article, we deeply investigate the squarefree powers of edge ideals of forests. Our result are based on the technique of \textit{Betti splitting}. In their groundbreaking article, Eliahou and Kervaire considered an earlier version of Betti splitting \cite[Proposition 3.1]{EK}. Later on, Francisco, H\`a and Van Tuyl, inspired by the results in  \cite{EK}, introduced the notion of Betti splitting \cite{FHT2009, Van2011}. Roughly speaking, if $I\subset S$ is a monomial ideal and $I_1$ and $I_2$ are sub--ideals of $I$ such that $\G(I)$ is the disjoint union of $\G(I_1)$ and $\G(I_2)$, then $I=I_1+I_2$ is said to be a Betti splitting, if the minimal free resolution of $I$ can be recovered from the minimal free resolutions of $I_1$, $I_2$ and $I_1\cap I_2$.
	
Our article is structured as follows. In Section \ref{Sec1:SqfreePow}, we consider Betti splittings of squarefree powers. Theorem \ref{Thm:CriterionBettiSplit} is a technical but powerful criterion that can be used to establish if a certain decomposition $I=I_1+I_2$ is a Betti splitting. This fact is used in Lemma \ref{Lemma:BettiSplit(I,x_n)} to reprove a result in \cite{FHH2022} under weaker assumptions (Proposition \ref{Prop:(I,x)sqfrPowers}).
	
Section \ref{Sec2:SqfreePow} deeply investigates squarefree powers of edge ideals of forests. For a vertex $v$ in $G$, we say $w$ is a \textit{neighbor} of $v$ if $\{v, w\}\in E(G)$. We denote the set of all neighbors of $v$ by $N_G(v)$. $N_G(v)$ is called the \textit{neighborhood} of $v$. The \textit{degree} of $v$, $\deg_G(v)$, is the number $|N_G(v)|$. If $\deg_G(v)=1$, $v$ is called a \textit{leaf}. Any forest has at least two leaves. A leaf $v\in V(G)$, with unique neighbor $w$, is called a \textit{distant leaf} if at most one of the neighbors of $w$ is not a leaf. The existence of a distant leaf in any forest $G$ (Proposition \ref{Prop:DistantLeaf}) provides a natural Betti splitting for $I(G)^{[k]}$ (Lemma \ref{Lem:I(G)ForestBettiSplit}). Then, using induction, we successfully prove that $g_{I(G)}$ is non-increasing, as stated in Theorem \ref{Thm:gITreeNonInc}, resolving a conjecture from \cite{EH2021}. We also compute the normalized depth function of the edge ideal of a path and use this result to obtain a general upper bound for the normalized depth function of the edge ideal of any graph.
	
In the last section, we compute the Castelnuovo--Mumford regularity $\reg(I(G)^{[k]})$ in terms of the combinatorics of the forest $G$ (Theorem \ref{Thm:ConjEreyHibi}), solving affirmatively a conjecture due to Erey and Hibi \cite[Conjecture 31]{EH2021}.

Finally, we would like to remark that not all squarefree monomial ideals admit a Betti splitting, as pointed out by Bolognini \cite[Example 4.6]{DB}, see also \cite[Example 4.2]{FHT2009}. On the other hand, Francisco, H\`a and Van Tuyl showed that the edge ideal %$I(G)$ 
of any graph %$G$ 
always admits a Betti splitting \cite[Corollary 3.1]{FHT2009}. It is an open question if all the squarefree powers of an edge ideal admit a Betti splitting.

\section{Betti splittings for squarefree powers}\label{Sec1:SqfreePow}
Let $S=K[x_1,\dots,x_n]$ be the standard graded polynomial ring over a field $K$. Let $I\subset S$ be a monomial ideal. By $\G(I)$ we denote the unique minimal set of monomial generators of $I$. Let $I_1,I_2\subset S$ monomial ideals such that $\G(I)$ is the disjoint union of $\G(I_1)$ and $\G(I_2)$. We say that $I=I_1+I_2$ is a \textit{Betti splitting} \cite[Definition 1.1]{FHT2009} if
\begin{equation}\label{eq:BettiSplitEq}
	\beta_{i,j}(I)=\beta_{i,j}(I_1)+\beta_{i,j}(I_2)+\beta_{i-1,j}(I_1\cap I_2),\ \ \ \text{for all}\ i,j\ge0.
\end{equation}

Consider the natural short exact sequence
$$
0\rightarrow I_1\cap I_2\rightarrow I_1\oplus I_2\rightarrow I\rightarrow0.
$$
Then (\ref{eq:BettiSplitEq}) holds if and only if the following induced maps in $\Tor$ of the above sequence
$$
\Tor_i^S(K,I_1\cap I_2)\rightarrow\Tor_i^S(K,I_1)\oplus\Tor_i^S(K,I_2)
$$
are zero for all $i\ge0$ \cite[Proposition 2.1]{FHT2009}. In this case, we say that the inclusion map $I_1\cap I_2\rightarrow I_1\oplus I_2$ is \textit{$\Tor$-vanishing}. See  \cite{AM19,NV19} for more details on this subject.

The next criterion gives a useful method to prove that a map $J\rightarrow L$ is $\Tor$-vanishing, where $(0)\ne J\subset L$ are monomial ideals of $S$. It appeared implicitly for the first time in \cite{EK} (see also the proof of \cite[Lemma 4.2]{NV19}).

\begin{theorem}{\normalfont(\cite[Proposition 3.1]{EK})}\label{Thm:CriterionBettiSplit}
	Let $J,L\subset S$ be non--zero monomial ideals with $J\subset L$. Suppose there exists a map $\varphi:\G(J)\rightarrow \G(L)$ such that for any $\emptyset\ne\Omega\subseteq \G(J)$ we have
	$$
	\lcm(u:u\in\Omega)\in \mathfrak{m}(\lcm(\varphi(u):u\in\Omega)),
	$$
	where $\mathfrak{m}=(x_1,x_2,\dots,x_n)$. Then the inclusion map $J\rightarrow L$ is $\Tor$-vanishing.
\end{theorem}

Let $I$ be a monomial ideal, following \cite{NV19} we denote by $\partial^* I$ the ideal generated by the elements of the form $f/x_i$, where $f$ is a minimal monomial generator of $I$ and $x_i$ is a variable dividing $f$. 

We recall the following result from \cite{NV19}.
\begin{lemma}\label{Lemma:EsistenzaVarphi}
	\textup{(\cite[Proposition 4.4]{NV19})} Let $J,L\subset S$ be non--zero monomial ideals. Suppose that $\partial^* J\subseteq L$. Then $J\subseteq\mathfrak{m}L$ and there exists a map $\varphi:\mathcal{G}(J)\rightarrow \mathcal{G}(L)$ satisfying the assumption of Theorem \ref{Thm:CriterionBettiSplit}. In particular, the inclusion map $J\rightarrow L$ is $\Tor$-vanishing.
\end{lemma}
Note that for all $1\le\ell<k\le\nu(I)$, we have $\partial^* I^{[k]}\subset I^{[k-1]}\subseteq I^{[\ell]}$. Hence, the previous lemma implies immediately the following result.
\begin{corollary}\label{Cor:I{[k]}partial}
	Let $I\subset S$ be a squarefree monomial ideal. Then the map $I^{[k]}\rightarrow I^{[\ell]}$ is $\Tor$-vanishing for all $1\le\ell<k\le\nu(I)$.
\end{corollary}

Let $I=I_1+I_2$ with $\G(I_2)=\{u\in \G(I):x\ \textit{divides}\ u\}$ and $\G(I_1)=\G(I)\setminus \G(I_2)$, for some variable $x$. We say that $I=I_1+I_2$ is an \textit{$x$-partition}. In addition, if it is a Betti splitting, we say that $I=I_1+I_2$ is an \textit{$x$-splitting} \cite[Definition 2.6]{FHT2009}.

\begin{lemma}\label{Lemma:BettiSplit(I,x_n)}
	Let $I\subset K[x_1,\dots,x_{n-1}]$ be a squarefree monomial ideal. Then, for all $1<k\le\nu(I)$
	\begin{equation}\label{eq:(I,x_n)}
		(I,x_n)^{[k]}=I^{[k]}+x_nI^{[k-1]}\subset S
	\end{equation}
	is a Betti splitting.
\end{lemma}
\begin{proof}
	Note that $(I,x_n)^{[k]}=I^{[k]}+x_nI^{[k-1]}$ is a $x_n$-partition. Moreover,
	$$
	I^{[k]}\cap x_nI^{[k-1]}=x_n[I^{[k]}\cap I^{[k-1]}]=x_nI^{[k]}.
	$$
	Thus, to prove that (\ref{eq:(I,x_n)}) is a Betti splitting, it remains to be shown that the inclusion map $x_nI^{[k]}\rightarrow I^{[k]}\oplus x_nI^{[k-1]}$ is $\Tor$-vanishing. That is, for all $i$ and all ${\bf a}=(a_1,\dots,a_n)\in\ZZ^n$, the map
	$$
	\Tor_i^S(K,x_nI^{[k]})_{\bf a}\rightarrow\Tor_i^S(K,I^{[k]})_{\bf a}\oplus\Tor_i^S(K,x_{n}I^{[k-1]})_{\bf a}
	$$
	is zero. But if $\Tor_i^S(K,x_nI^{[k]})_{\bf a}\ne0$, then $a_n>0$ since $x_n$ divides all minimal generators of $x_nI^{[k]}$. On the other hand, if $a_n>0$, then $\Tor_i^S(K,I^{[k]})_{\bf a}=0$ because $x_n$ does not divide any generator of $I^{[k]}$. Therefore, the map $x_nI^{[k]}\rightarrow I^{[k]}\oplus x_nI^{[k-1]}$ is $\Tor$-vanishing if and only if for all $i$ and all ${\bf a}=(a_1,\dots,a_n)\in\ZZ^n$ with $a_n>0$, the map $\Tor_i^S(K,x_nI^{[k]})_{\bf a}\rightarrow\Tor_i^S(K,x_{n}I^{[k-1]})_{\bf a}$
	is zero. This is equivalent to saying that $x_nI^{[k]}\rightarrow x_nI^{[k-1]}$ is $\Tor$-vanishing. This latter condition holds if and only if the map $I^{[k]}\rightarrow I^{[k-1]}$ is such. But this is obviously the case by Corollary \ref{Cor:I{[k]}partial}.
\end{proof}

The next result was proved in \cite[Proposition 2.4]{FHH2022} under the additional assumption that all squarefree powers $I^{[k]}$ are componentwise linear. This assumption was made to ensure that $I^{[k]}+x_{n}I^{[k-1]}$ is a Betti splitting \cite[Theorem 3.3]{DB}. Thanks to Lemma \ref{Lemma:BettiSplit(I,x_n)}, this hypothesis can be removed. Therefore, we have
\begin{proposition}\label{Prop:(I,x)sqfrPowers}
	Let $I\subset K[x_1,\dots,x_{n-1}]$ be a squarefree monomial ideal. Set $J=(I,x_n)$ and $d_k=\textup{indeg}(I^{[k]})$ for $1\le k\le\nu(I)$, $g_I(0)=g_I(\nu(I)+1)=+\infty$ and $d_0=0$. Then $\nu(J)=\nu(I)+1$ and for all $1\le k\le\nu(J)$,
	$$
	g_J(k)\ =\ \min\{g_I(k)+d_k-d_{k-1}-1,g_I(k-1)\}.
	$$
\end{proposition}
\section{Squarefree powers of forests}\label{Sec2:SqfreePow}
A graph $G$ is called a \textit{forest} if it is acyclic. A connected forest is called a \textit{tree}. A \textit{leaf} $v$ of a graph $G$ is a vertex incident to only one edge. Any tree possesses at least two leaves. Let $v\in V(G)$ be a leaf and $w$ be the unique neighbor of $v$. Following \cite{EH2021}, we say that $v$ is a \textit{distant leaf} if at most one of the neighbors of $w$ is not a leaf. In this case, we say that $\{w,v\}$ is a \textit{distant edge}. The following picture displays this situation: %The gray area represents the graph $G\setminus\{w\}$.}
\vspace*{-0.3cm}

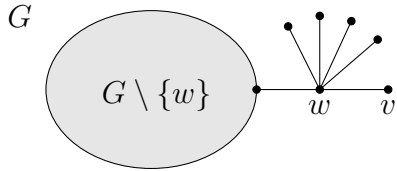
\begin{figure}[H]%\label{figure}
\centering
\begin{tikzpicture}[scale=0.7]
	\filldraw[fill=black!10!white] (0,0) ellipse (2cm and 1.5cm);
	\filldraw (2,0) circle (2pt);
	\filldraw (3.2,0) circle (2pt) node[below]{$w$};
	\filldraw (4.5,0) circle (2pt) node[below]{$v$};
	\filldraw (2.6,1.2) circle (2pt);
	\filldraw (3.2,1.4) circle (2pt);
	\filldraw (3.8,1.3) circle (2pt);
	\filldraw (4.3,0.95) circle (2pt);
	\draw[-] (3.2,0)--(2,0);
	\draw[-] (4.5,0)--(3.2,0);
	\draw[-] (3.2,0)--(2.6,1.2);
	\draw[-] (3.2,0)--(3.2,1.4);
	\draw[-] (3.2,0)--(3.8,1.3);
	\draw[-] (3.2,0)--(4.3,0.95);
	\filldraw (-2.5,1.8) node[below]{$G$};
	\filldraw (0.1,0.4) node[below]{$G\setminus\{w\}$};
\end{tikzpicture}
\caption{A forest $G$ with distant edge $\{w,v\}$}\label{figure}
\end{figure}\vspace*{-0.15cm}
\noindent The gray area represents the graph $G\setminus\{w\}$.
Next picture clarifies Figure \ref{figure}.
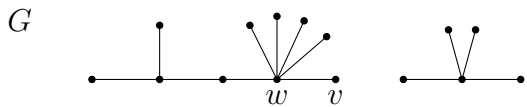
\begin{figure}[H]
	\centering
	\begin{tikzpicture}[scale=0.6]
		\filldraw (2,0) circle (2pt);
		\filldraw (-0.9,0) circle (2pt);
		\filldraw (0.6,0) circle (2pt);
		\filldraw (0.6,1.2) circle (2pt);
		\filldraw (3.2,0) circle (2pt) node[below]{$w$};
		\filldraw (4.5,0) circle (2pt) node[below]{$v$};
		\filldraw (2.6,1.2) circle (2pt);
		\filldraw (3.2,1.4) circle (2pt);
		\filldraw (3.8,1.3) circle (2pt);
		\filldraw (4.3,0.95) circle (2pt);
		\draw[-] (-0.9,0)--(3.2,0);
		\draw[-] (4.5,0)--(3.2,0);
		\draw[-] (3.2,0)--(2.6,1.2);
		\draw[-] (3.2,0)--(3.2,1.4);
		\draw[-] (3.2,0)--(3.8,1.3);
		\draw[-] (3.2,0)--(4.3,0.95);
		\filldraw (6,0) circle (2pt);
		\filldraw (7.3,0) circle (2pt);
		\filldraw (7,1.1) circle (2pt);
		\filldraw (7.6,1.1) circle (2pt);
		\filldraw (8.6,0) circle (2pt);
		\draw[-] (6,0) -- (8.6,0);
		\draw[-] (7.3,0) -- (7,1.1);
		\draw[-] (7.3,0) -- (7.6,1.1);
		\draw[-] (0.6,0) -- (0.6,1.2);
		\filldraw (-2.5,1.8) node[below]{$G$};
	\end{tikzpicture}
\end{figure}\vspace*{-0.15cm}

It describes a forest $G$ with $\{w,v\}$ as a distant edge.

Our goal is to compute explicitly the regularity function $\reg(I(G)^{[k]})$ and to prove that the normalized depth function $g_{I(G)}(k)$ of $I(G)$ is a non-increasing function when $G$ is a forest. We state our first result below and provide its proof after a few auxiliary results.
\begin{theorem}\label{Thm:gITreeNonInc}
	Let $G$ be a forest. Then $g_{I(G)}$ is non-increasing.
\end{theorem}

The previous theorem and the results in this section are based upon the following proposition. It already appeared in \cite[Proposition 9.1.1]{J2004} and \cite[Lemma 20]{EH2021}. But for the sake of completeness we include a different shorter proof.
\begin{proposition}\label{Prop:DistantLeaf}
	Let $G$ be a forest. Then $G$ has a distant leaf.
\end{proposition}
\begin{proof}
	It is enough to assume that $G$ is a tree. Let $v_1,v_2,\dots,v_r$ be an induced path of $G$ of maximal length. We claim that $v_{r}$ is a distant leaf. Indeed, let $N_G(v_{r-1})=\{w_1,\dots,w_s\}$. Up to a relabeling, we may assume $w_1=v_{r-2}$ and $w_2=v_r$. Since $\deg_G(w_2)=1$, it is enough to show that $\deg_G(w_i)=1$ for $i=3,\dots,s$. Suppose for a contradiction that this is not the case. Thus $\deg_G(w_i)>1$ for some $i\in\{3,\dots,s\}$. Let $u$ be a neighbor of $w_i$ different from $v_{r-1}$. Then, $v_1,v_2,\dots,v_{r-1},w_i,u$ is an induced path of $G$ of length $r+1$, a contradiction. The assertion follows.
\end{proof}

We fix the following setup. We will consider forests $G$ satisfying $\nu(G)\ge3$. The cases $\nu(G)=1,2$ will be addressed directly in the proof of Theorem \ref{Thm:gITreeNonInc}.
\begin{setup}\label{SetupG}
	Let $G$ be a forest with the vertex set $[n]$ and $\nu(G)\ge3$. Let $n\in V(G)$ be a distant leaf with distant edge $\{n-1,n\}\in E(G)$ such that $N_G(n-1)=\{i_1,\dots,i_t,n-2,n\}$ with $t\ge0$, $\deg_G(i_j)=1$ for $j=1,\dots,t$, $\deg_G(n)=1$ and $\deg_G(n-2)\ge1$. Let $G_1$, $G_2$ and $G_3$ be the induced subgraphs of $G$ on the vertex sets $[n-1]$, $[n-2]$ and $[n-3]$, respectively. Note that $|V(G_1)|,|V(G_2)|,|V(G_3)|<n$, and moreover $I(G_3)\ne0$ because $\nu(G)\ge3$. Indeed, if $I(G_3)=0$, then either $G_3$ contains no vertices, or all vertices of $G_3$ are adjacent to $n-2$. Then,  $\{n-2,n-1\}\in E(G)$ and the other edges $\{i,j\}$ of $G$ have $i=n-2$ or $i=n-1$. Then, it is clear that $\nu(G)\le2$, against our assumption.%$E(G)$ would consist of the edge $\{n-2,n-1\}$ and the other edges $\{i,j\}$ have $i=n-2$ or $i=n-1$. Then, it is clear that $\nu(G)\le2$, against our assumption.}
\end{setup}

\begin{lemma}\label{Lem:I(G)ForestBettiSplit}
	Assume Setup \ref{SetupG}. Then, for all $1\leq k\leq \nu(G)$,
	\begin{equation}\label{eq:I(G)BettiSplitGForest}
		I(G)^{[k]}=I(G_1)^{[k]}+x_nx_{n-1}I(G_2)^{[k-1]}
	\end{equation}
	is a Betti splitting.
\end{lemma}
\begin{proof}
	Firstly, we prove that (\ref{eq:I(G)BettiSplitGForest}) holds. The sum of the right-hand side is evidently a subset of $I(G)^{[k]}$. Hence, we only need to establish the reverse inclusion. Indeed, let $u=e_1\cdots e_k\in \G(I(G)^{[k]})$ where $e_j=x_{p_j}x_{q_j}$ for $j=1,\dots,k$. Then $M=\{\{p_j,q_j\}:j=1,\dots,k\}$ is a $k$-matching of $G$. If $n\in V(M)$, then $\{p_j,q_j\}=\{n-1,n\}$ for some $j$, because $n$ is a leaf. Then $M'=M\setminus\{\{n-1,n\}\}$ is a $(k-1)$-matching of $G_2$. Hence $u\in \G(x_{n}x_{n-1}I(G_2)^{[k-1]})$ in this case. Otherwise, assume that $n\notin V(M)$, then $M$ is a $k$-matching of $G_1$ and $u\in \G(I(G_1)^{[k]})$. These two cases show that (\ref{eq:I(G)BettiSplitGForest}) holds.
	Note that (\ref{eq:I(G)BettiSplitGForest}) is an $x_n$-partition because $\G(I(G)^{[k]})$ is the disjoint union of $\G(I(G_1)^{[k]})$ and $\G(x_{n}x_{n-1}I(G_2)^{[k-1]})$. 
	
	Next, we set $J=I(G_1)^{[k]}\cap x_nx_{n-1}I(G_2)^{[k-1]}$ and claim that
	\begin{equation}\label{eq:intersectBettiSplitGForest}
		J=x_nx_{n-1}\big[I(G_3)^{[k]}+x_{n-2}I(G_3)^{[k-1]}+\sum_{j=1}^{t}x_{i_j}I(G_2)^{[k-1]}\big].
	\end{equation}
	To prove the above equality, we first show that
	\begin{equation}\label{eq:I(G_1)BettiSplit}
		I(G_1)^{[k]}=x_{n-1}x_{n-2}I(G_3)^{[k-1]}+\sum_{j=1}^tx_{n-1}x_{i_j}I(G_2)^{[k-1]}+I(G_2)^{[k]}.
	\end{equation}
	Let $u=e_1\cdots e_k\in \G(I(G_1)^{[k]})$ with $e_j=x_{p_j}x_{q_j}$ for $j=1,\dots,k$. This means that $M=\{\{p_j,q_j\}:j=1,\dots,k\}$ is a $k$-matching of $G_1$. If $n-1\notin V(M)$, none of the vertices from $\{i_1,\dots,i_t\}$ belongs to $V(M)$. Thus $u\in \G(I(G_2)^{[k]})$. Otherwise, if $n-1\in V(M)$, then either $\{n-1,n-2\}\in M$ or $\{n-1,i_j\}\in M$ for some $j\in[t]$. In the first case, $M'=M\setminus\{\{n-1,n-2\}\}$ is a $(k-1)$-matching of $G_3$. Then $u\in \G(x_{n-1}x_{n-2}I(G_3)^{[k-1]})$. In the second case, $M'=M\setminus\{\{n-1,i_j\}\}$ is a $(k-1)$-matching of $G_2$ and $u\in \G(x_{n-1}x_{i_j}I(G_2)^{[k-1]})$. These cases show that equation (\ref{eq:I(G_1)BettiSplit}) holds.
	
	Now, we proceed to prove (\ref{eq:intersectBettiSplitGForest}). We apply equation (\ref{eq:I(G_1)BettiSplit}) to show the followings:\smallskip
	\begin{enumerate}
		\item[(a)] $x_{n-1}x_{n-2}I(G_3)^{[k-1]}\cap x_{n-1}x_nI(G_2)^{[k-1]}=x_nx_{n-1}x_{n-2}I(G_3)^{[k-1]}$;\vspace*{2mm}
		\item[(b)] $x_{n-1}x_{i_j}I(G_2)^{[k-1]}\cap x_{n-1}x_nI(G_2)^{[k-1]}=x_nx_{n-1}x_{i_j}I(G_2)^{[k-1]}$, for $j\in[t]$;\vspace*{2mm}
		\item[(c)] $I(G_2)^{[k]}\cap x_{n-1}x_nI(G_2)^{[k-1]}=x_nx_{n-1}I(G_2)^{[k]}$.
	\end{enumerate}
	Equation (a) follows from the inclusions $x_{n-2}I(G_3)^{[k-1]}\!\subset I(G_3)^{[k-1]}\!\subset\! I(G_2)^{[k-1]}$ and the fact that $x_n$ does not divide any generator of $I(G_3)$ and $I(G_2)$. Equation (b) is justified similarly, given that $x_n$, and $x_{i_j}$ do not divide any generator of $I(G_2)$. Lastly, equation (c) is derived from the containment $I(G_2)^{[k]}\subset I(G_2)^{[k-1]}$ and the fact that none of the generators of $I(G_2)$ are divisible by $x_{n-1}$ and $x_n$.\\
	In light of equations (a), (b) and (c), we deduce equation (\ref{eq:I(G_1)BettiSplit}) by establishing the following:
	\begin{equation}\label{eq:I(G_2)I(G_3)}
		I(G_2)^{[k]}+x_{n-2}I(G_3)^{[k-1]}=I(G_3)^{[k]}+x_{n-2}I(G_3)^{[k-1]}
	\end{equation}
	The inclusion ``$\supseteq$" is clear. For the other inclusion, consider $u=e_1\cdots e_k\in \G(I(G_2)^{[k]})$ such that $e_i\in \G(I(G_2))$ for each $i$. If $x_{n-2}$ divides $u$, then $x_{n-2}$ divides $u_i$ for some $i$. Without loss of generality, we may assume $i=1$. This implies that $x_{n-2}u_2\cdots u_k$ belongs to $\G(x_{n-2}I(G_3)^{[k-1]})$. On the other hand, if $x_{n-2}$ does not divide $u$, then $u$ is an element of $\G(I(G_3)^{[k]})$, as required. Hence, formula (\ref{eq:I(G_2)I(G_3)}) is established.
	(a), (b) and (c) combined with (\ref{eq:I(G_1)BettiSplit}) and (\ref{eq:I(G_2)I(G_3)}) imply the desired formula (\ref{eq:intersectBettiSplitGForest}).\\
	To conclude that (\ref{eq:I(G)BettiSplitGForest}) is indeed a Betti splitting, we must show that the inclusion map $J\rightarrow I(G_1)^{[k]}\oplus x_nx_{n-1}I(G_2)^{[k-1]}$ is $\Tor$-vanishing. Observe that, while $x_n$ does not divide any generator of $I(G_1)^{[k]}$, it divides all generators of $J$ according to (\ref{eq:intersectBettiSplitGForest}). Therefore, arguing as in the proof of Lemma \ref{Lemma:BettiSplit(I,x_n)}, it is enough to show that the inclusion map $J\rightarrow x_nx_{n-1}I(G_2)^{[k-1]}$ is $\Tor$-vanishing.\\
	For this purpose, we define the map
	$$
	\varphi:\G(J)\rightarrow \G(x_nx_{n-1}I(G_2)^{[k-1]})
	$$
	considering the following three cases:
	\begin{enumerate}
		\item[(i)] Let $x_nx_{n-1}u\in \G(x_nx_{n-1}I(G_3)^{[k]})$. By Lemma \ref{Lemma:EsistenzaVarphi}, there exists a map $\widetilde{\varphi}:\G(I(G_3)^{[k]})\rightarrow \G(I(G_3)^{[k-1]})$ verifying Theorem \ref{Thm:CriterionBettiSplit}. Then, we have $\widetilde{\varphi}(u)\in \G(I(G_3)^{[k-1]})\subset \G(I(G_2)^{[k-1]})$ and we set $\varphi(x_nx_{n-1}u)=x_nx_{n-1}\widetilde{\varphi}(u)$.
		\item[(ii)] Let $x_nx_{n-1}x_{n-2}u\in \G(x_nx_{n-1}x_{n-2}I(G_3)^{[k-1]})$. Then, $u\in \G(I(G_2)^{[k-1]})$ and we set $\varphi(x_nx_{n-1}x_{n-2}u)=x_nx_{n-1}u$.
		\item[(iii)] Let $x_nx_{n-1}x_{i_j}u\in \G(x_nx_{n-1}x_{i_j}I(G_2)^{[k-1]})$, for some $j\in[t]$. Then, we set $\varphi(x_nx_{n-1}x_{i_j}u)=x_nx_{n-1}u$.
	\end{enumerate}
	The map $\varphi$ is well--defined by equation (\ref{eq:intersectBettiSplitGForest}). By Theorem \ref{Thm:CriterionBettiSplit}, it suffices to show that for any subset $\Omega\subseteq \G(J)$,
	\begin{equation}\label{eq:ContainOmega}
		\lcm(u:u\in\Omega)\in \mathfrak{m}(\lcm(\varphi(u):u\in\Omega)),
	\end{equation}
	where $\mathfrak{m}=(x_1,\dots,x_n)$ is the maximal ideal of $S$. Let $\Omega\subseteq \G(J)$. Then
	$$
	\Omega=\big(\bigcup_{\ell}\{x_nx_{n-1}u_\ell\}\big)\cup \big(\bigcup_{r}\{x_{n}x_{n-1}x_{n-2}w_r\}\big)\cup \big(\bigcup_{j=1}^{t}\bigcup_{s} \{x_{n}x_{n-1}x_{i_j}z_{j,s}\}\big).
	$$
	where $u_\ell\in \G(I(G_3)^{[k]})$, $w_r\in \G(I(G_3)^{[k-1]})$ and $z_{j,s}\in \G(I(G_2)^{[k-1]})$. 
	
	It follows that %Thus,
	\begin{align*}
		\lcm(u:u\in\Omega)&=x_nx_{n-1}\lcm(u_\ell,x_{n-2}w_r,x_{i_j}z_{j,s}:\ell,r,s,\ j\in[t])
	\end{align*}
	and %Whereas,
	\begin{align*}
		\lcm(\varphi(u):u\in\Omega)=&\ x_nx_{n-1}\lcm(\widetilde{\varphi}(u_\ell),w_r,z_{j,s}:\ell,r,s,\ j\in[t]).
	\end{align*}
	From these expressions, we deduce that $\lcm(u:u\in\Omega)\in \mathfrak{m}(\lcm(\varphi(u):u\in\Omega))$. Specifically, if at least one $w_r$ or at least one $z_{j,s}$ appears in $\Omega$, then either $x_{n-2}$ divides $\lcm(u:u\in\Omega)$ but does not divide $\lcm(\varphi(u):u\in\Omega)$, or $x_{i_j}$ divides $\lcm(u:u\in\Omega)$ but does not divide $\lcm(\varphi(u):u\in\Omega)$. Conversely, if neither $w_r$ nor $z_{j,s}$ appears in $\Omega$, the $\lcm(u:u\in\Omega)$ is strictly divided by $\lcm(\varphi(u):u\in\Omega)$, due to the conditions met by $\widetilde{\varphi}$ as described in Theorem \ref{Thm:CriterionBettiSplit}. This concludes our proof.
\end{proof}

Now, consider equation (\ref{eq:intersectBettiSplitGForest}). Set
\begin{align*}
J_1&=x_{n}x_{n-1}[I(G_3)^{[k]}+x_{n-2}I(G_3)^{[k-1]}],\\
J_2&=x_{n}x_{n-1}\sum_{j=1}^tx_{i_j}I(G_2)^{[k-1]}=x_{n}x_{n-1}(x_{i_1},\dots,x_{i_t})I(G_2)^{[k-1]}.
\end{align*}
With this notation, $J=J_1+J_2$. Note further that $J_2\ne(0)$ if and only if $t>0$.
\begin{lemma}\label{Lemma:BettiSplittingJ=J1+J2}
	Assume Setup \ref{SetupG} and suppose $t>0$. With the notation above, we have
	\begin{equation}\label{eq:intersectIntersect}
		J_1\cap J_2=(x_{i_1},\dots,x_{i_t})J_1
	\end{equation}
	and $J=J_1+J_2$ is a Betti splitting.
\end{lemma}
\begin{proof}
	We first show that (\ref{eq:intersectIntersect}) holds. This equality follows from the observation that none of the generators of $J_1$ are divisible by $x_{i_j}$ for any $j\in[t]$ and the following chain of containments:
	$$
	I(G_3)^{[k]}+x_{n-2}I(G_3)^{[k-1]}\subset I(G_2)^{[k]}+x_{n-2}I(G_2)^{[k-1]}\subset I(G_2)^{[k-1]}.
	$$
	It remains to prove that the map $J_1\cap J_2\rightarrow J_1\oplus J_2$ is $\Tor$-vanishing. Note that each generator of $J_1\cap J_2$ is divisible by $x_{i_j}$ for some $j\in[t]$, while none of the generators of $J_1$ are divisible by any of the $x_{i_j}$. Hence, arguing as in Lemma \ref{Lemma:BettiSplit(I,x_n)}, it is enough to show that the inclusion $J_1\cap J_2\rightarrow J_2$ is $\Tor$-vanishing. For this purpose, we apply Theorem \ref{Thm:CriterionBettiSplit} again. Let 
	$$
	\varphi:\G((x_{i_1},\dots,x_{i_t})J_1)\rightarrow \G(J_2)
	$$ 
	be defined considering the following cases:
	\begin{enumerate}
		\item[(i)] Let $x_nx_{n-1}x_{i_j}u\in \G((x_{i_1},\dots,x_{i_t})J_1)$ with $u\in \G(I(G_3)^{[k]})$ and $j\in[t]$. By Lemma \ref{Lemma:EsistenzaVarphi}, there exists $\widetilde{\varphi}:\G(I(G_3)^{[k]})\rightarrow \G(I(G_3)^{[k-1]})$ satisfying Theorem \ref{Thm:CriterionBettiSplit}. Then we set $\varphi(x_nx_{n-1}x_{i_j}u)=x_nx_{n-1}x_{i_j}\widetilde{\varphi}(u)$.
		\item[(ii)] Let $x_nx_{n-1}x_{n-2}x_{i_j}u\in \G((x_{i_1},\dots,x_{i_t})J_1)$ with $u\in \G(I(G_3)^{[k-1]})$, $j\in[t]$. Then $u\in \G(I(G_2)^{[k-1]})$ and we set $\varphi(x_nx_{n-1}x_{n-2}x_{i_j}u)=x_nx_{n-1}x_{i_j}u$.
	\end{enumerate}
	The map $\varphi$ is well defined by (\ref{eq:intersectIntersect}). Arguing as in the proof of Lemma \ref{Lem:I(G)ForestBettiSplit}, we see that $\varphi$ verifies the condition in Theorem \ref{Thm:CriterionBettiSplit}, as desired.
\end{proof}

For the proof of the next corollary, we recall two general facts that we will use repeatedly in what follows. Let $u\in S$ be a monomial and $L\subset S$ be a monomial ideal, then $\depth(S/uL)=\depth(S/L)$. Moreover, if $\{x_{p_1},\dots,x_{p_s}\}$ is a collection of variables not dividing any generator of $L$, then \cite[Corollary 3.2]{HRR}
$$
\depth(S/(x_{p_1},\dots,x_{p_s})L)=\depth(S/L)-(s-1).
$$

The next remark will be crucial for the sequel.

\begin{remark}
	Let $I\subset S$ be a squarefree monomial ideal, where $S=K[x_1,\dots,x_n]$ and consider the ring extension $S'=S[y_1,\dots,y_m]=K[x_1,\dots,x_n,y_1,\dots,y_m]$. Then $\nu(I)=\nu(IS')$. With abuse of notation, we denote again by $I$ the ideal $IS'$, that is, the extension of $I$ in $S'$. Note that $\depth(S'/I)=\depth(S/I)+m$. Let $g_I$, $h_I$ be the normalized depth functions of $I\subset S$, and $I\subset S'$, respectively. Then, $g_I(k)=h_I(k)-m$ for all $1\le k\le\nu(I)$. Thus $g_I$ is non-increasing if and only if $h_I$ is non-increasing. 
\end{remark}

To avoid unnecessary distinctions, we will regard $I(G),I(G_1),I(G_2)$ and $I(G_3)$ as ideals of $S=K[x_1,\dots,x_n]$. Therefore, $g_{I(G_i)}(k)=\depth(S/I(G_i)^{[k]})-(2k-1)$ for all $i$ and $k$.

\begin{corollary}\label{cor:gIRecursiveFormula}
	Assume Setup \ref{SetupG}. 
	\begin{enumerate}
		\item[\em(a)] If $t=0$, then
		\begin{equation}\label{eq:gI(G)TreeRecursive'}
			\begin{aligned}
				g_{I(G)}(k)=\min\{&g_{I(G_1)}(k),g_{I(G_2)}(k-1)-2,\\&g_{I(G_3)}(k-1)-3,g_{I(G_3)}(k)-2\}
			\end{aligned}
		\end{equation}
		for all $1\le k \le \nu(G)$.\medskip
		\item[\em(b)] If $t>0$, then
		\begin{equation}\label{eq:gI(G)TreeRecursive}
			\begin{aligned}
				g_{I(G)}(k)=\min\{&g_{I(G_1)}(k),\,g_{I(G_2)}(k-1)-2-t,\\&g_{I(G_3)}(k-1)-2-t,\,g_{I(G_3)}(k)-1-t\}
			\end{aligned}
		\end{equation}
		for all $1\le k \le \nu(G)$.
	\end{enumerate}
\end{corollary}
\begin{proof}
	First, we compute $\depth(S/I(G)^{[k]})$. By Lemma \ref{Lem:I(G)ForestBettiSplit}, the decomposition (\ref{eq:I(G)BettiSplitGForest}) is a Betti splitting. Hence, from \cite[Corollary 2.2]{FHT2009} and \cite[Corollary A.4.3]{JT},
	$$
	\depth(S/I(G)^{[k]})=\min\{\depth(S/I(G_1)^{[k]}),\depth(S/I(G_2)^{[k-1]}),\depth(S/J)-\!1\},
	$$
	where $J$ is given in (\ref{eq:intersectBettiSplitGForest}). Now we distinguish two cases. \medskip\\
	\textsc{Case 1.} Let $t=0$. Then $J=J_1=x_{n}x_{n-1}[I(G_3)^{[k]}+x_{n-2}I(G_3)^{[k-1]}]$ is a Betti splitting by Lemma \ref{Lemma:BettiSplit(I,x_n)}. Since $I(G_3)^{[k]}\cap x_{n-2}I(G_3)^{[k-1]}=x_{n-2}I(G_3)^{[k]}$, we have
	\begin{align*}
		\depth(S/J)&=\min\{\depth(S/I(G_3)^{[k]}),\depth(S/I(G_3)^{[k-1]}),\depth(S/I(G_3)^{[k]})-1\}\\
		&=\min\{\depth(S/I(G_3)^{[k-1]}),\depth(S/I(G_3)^{[k]})-1\}.
	\end{align*}
	We have
	\begin{align*}
		\depth(S/I(G_1)^{[k]})-\!(2k-1)&=g_{I(G_1)}(k),\\
		\depth(S/I(G_2)^{[k-1]})-\!(2k-1)&=\depth(S/I(G_2)^{[k-1]})-(2(k-1)-1)-2
		\\&=g_{I(G_2)}(k-1)-2.
	\end{align*}
	Similar computations yield $\depth(S/I(G_3)^{[k-1]})-(2k-1)=g_{I(G_3)}(k-1)-2$ and $\depth(S/I(G_3)^{[k]})-(2k-1)-1=g_{I(G_3)}(k)-1$. Hence, (\ref{eq:gI(G)TreeRecursive'}) holds in this case.
	\medskip\\
	\textsc{Case 2.} Let $t>0$.
	Now we compute $\depth(S/J)$. By Lemma \ref{Lemma:BettiSplittingJ=J1+J2}, $J=J_1+J_2$ is a Betti splitting and $J_1\cap J_2=(x_{i_1},\dots,x_{i_t})J_1$. Thus,
	\begin{align}\label{eq:case2}
		\nonumber \depth(S/J)-1&=\min\{\depth(S/J_1)-\!1,\!\depth(S/J_2)-\!1,\!\depth(S/J_1)-\!(t-1)-\!1\}\\
		&=\min\{\depth(S/I(G_2)^{[k-1]})-t,\depth(S/J_1)-t\}.
	\end{align}
	Moreover, from \textsc{Case 1}
	$$
	\depth(S/J_1)-t=\min\{\depth(S/I(G_3)^{[k-1]})-t,\depth(S/I(G_3)^{[k]})-t-1\}.
	$$
	Thus, from (\ref{eq:case2}), $\depth(S/I(G)^{[k]})$ is equal to
	\begin{align*}
		&\ \min\{\depth(S/I(G_1)^{[k]}),\depth(S/I(G_2)^{[k-1]}),\depth(S/J)-1\}\\
		=&\ \min\{\depth(S/I(G_1)^{[k]}),\depth(S/I(G_2)^{[k-1]}),\depth(S/I(G_2)^{[k-1]})-t,\\
		&\phantom{\min\{.}\depth(S/I(G_3)^{[k-1]})-t,\depth(S/I(G_3)^{[k]})-t-1\}\\
		=&\ \min\{\depth(S/I(G_1)^{[k]}),\depth(S/I(G_2)^{[k-1]})-t,\depth(S/I(G_3)^{[k-1]})-t,\\
		&\phantom{\min\{.}\depth(S/I(G_3)^{[k]})-t-1\}.
	\end{align*}
	We have $\depth(S/I(G_1)^{[k]})-(2k-1)=g_{I(G_1)}(k)$.	Moreover, 
	\begin{align*}
		\depth(S/I(G_2)^{[k-1]})-(2k-1)-t\ =g_{I(G_2)}(k-1)-2-t.
	\end{align*}
	Similarly,
	\begin{align*}
		\depth(S/I(G_3)^{[k-1]})-(2k-1)-t\ &=\ g_{I(G_3)}(k-1)-2-t,\\
		\depth(S/I(G_3)^{[k]})-(2k-1)-t-1\ &=\ g_{I(G_3)}(k)-1-t.
	\end{align*}
	Finally, we see that equation (\ref{eq:gI(G)TreeRecursive}) holds.
\end{proof}

The next example clarifies Corollary \ref{cor:gIRecursiveFormula}
\begin{example}\rm 
	Consider the graph $G$ on eleven vertices depicted below:
	
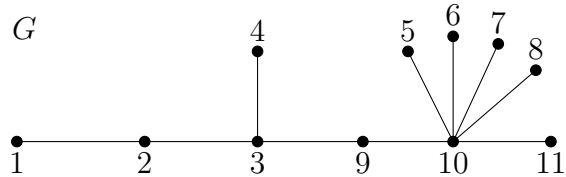
\begin{figure}[H]
		\centering
		\begin{tikzpicture}
			\filldraw (2,0) circle (2pt) node[below]{$9$};
			\filldraw (-0.9,0) circle (2pt) node[below]{$2$};
			\filldraw (-2.6,0) circle (2pt) node[below]{$1$};
			\filldraw (0.6,0) circle (2pt) node[below]{$3$};
			\filldraw (0.6,1.2) circle (2pt) node[above]{$4$};
			\filldraw (3.2,0) circle (2pt) node[below]{$10$};
			\filldraw (4.5,0) circle (2pt) node[below]{$11$};
			\filldraw (2.6,1.2) circle (2pt) node[above]{$5$};
			\filldraw (3.2,1.4) circle (2pt) node[above]{$6$};
			\filldraw (3.8,1.3) circle (2pt) node[above]{$7$};
			\filldraw (4.3,0.95) circle (2pt) node[above]{$8$};
			\draw[-] (-2.6,0)--(3.2,0);
			\draw[-] (4.5,0)--(3.2,0);
			\draw[-] (3.2,0)--(2.6,1.2);
			\draw[-] (3.2,0)--(3.2,1.4);
			\draw[-] (3.2,0)--(3.8,1.3);
			\draw[-] (3.2,0)--(4.3,0.95);
			\draw[-] (0.6,0) -- (0.6,1.2);
			\filldraw (-2.5,1.8) node[below]{$G$};
		\end{tikzpicture}
	\end{figure}
	Notice that $\nu(G)=3$ and $\{9,10\}$ is a distance edge of $G$. Using the notation in Setup \ref{SetupG}, with $n=11$, we have $N_G(10)=\{i_1,\dots,i_t,n-2,n\}=\{5,6,7,8,9,11\}$, with $t=4>0$. The graphs $G_1$, $G_2$, $G_3$ are depicted below:
	
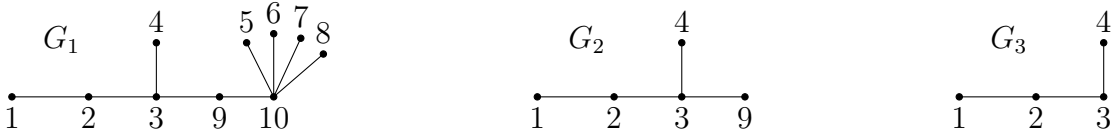
\begin{figure}[H]
		\centering
		\begin{tikzpicture}[scale=0.6]
	    	\filldraw (2,0) circle (2pt) node[below]{$9$};
	    	\filldraw (-0.9,0) circle (2pt) node[below]{$2$};
	    	\filldraw (-2.6,0) circle (2pt) node[below]{$1$};
	    	\filldraw (0.6,0) circle (2pt) node[below]{$3$};
	    	\filldraw (0.6,1.2) circle (2pt) node[above]{$4$};
	    	\filldraw (3.2,0) circle (2pt) node[below]{$10$};
	    	\filldraw (2.6,1.2) circle (2pt) node[above]{$5$};
	    	\filldraw (3.2,1.4) circle (2pt) node[above]{$6$};
	    	\filldraw (3.8,1.3) circle (2pt) node[above]{$7$};
	    	\filldraw (4.3,0.95) circle (2pt) node[above]{$8$};
	    	\draw[-] (-2.6,0)--(3.2,0);
	    	\draw[-] (3.2,0)--(2.6,1.2);
	    	\draw[-] (3.2,0)--(3.2,1.4);
	    	\draw[-] (3.2,0)--(3.8,1.3);
	    	\draw[-] (3.2,0)--(4.3,0.95);
	    	\draw[-] (0.6,0) -- (0.6,1.2);
	    	\filldraw (-1.5,1.8) node[below]{$G_1$};
	    \end{tikzpicture}
	    \hfill
		\begin{tikzpicture}[scale=0.6]
    		\filldraw (2,0) circle (2pt) node[below]{$9$};
    		\filldraw (-0.9,0) circle (2pt) node[below]{$2$};
    		\filldraw (-2.6,0) circle (2pt) node[below]{$1$};
    		\filldraw (0.6,0) circle (2pt) node[below]{$3$};
    		\filldraw (0.6,1.2) circle (2pt) node[above]{$4$};
    		\draw[-] (-2.6,0)--(2,0);
    		\draw[-] (0.6,0) -- (0.6,1.2);
    		\filldraw (-1.5,1.8) node[below]{$G_2$};
    	\end{tikzpicture}
    	\hfill
        \begin{tikzpicture}[scale=0.6]
        	\filldraw (-0.9,0) circle (2pt) node[below]{$2$};
        	\filldraw (-2.6,0) circle (2pt) node[below]{$1$};
        	\filldraw (0.6,0) circle (2pt) node[below]{$3$};
        	\filldraw (0.6,1.2) circle (2pt) node[above]{$4$};
        	\draw[-] (-2.6,0)--(0.6,0);
        	\draw[-] (0.6,0) -- (0.6,1.2);
        	\filldraw (-1.5,1.8) node[below]{$G_3$};
        \end{tikzpicture}
	\end{figure}
	In the previous picture we have not drawn the isolated vertices in $G_2$ and $G_3$.
	     
	Let $k=2$. According to Corollary \ref{cor:gIRecursiveFormula}(b), since $t=4$, we have
	\begin{align*}
	g_{I(G)}(2)\ &=\ \min\{g_{I(G_1)}(2),\,g_{I(G_2)}(1)-2-t,g_{I(G_3)}(1)-2-t,\,g_{I(G_3)}(2)-1-t\}\\
	 	    &=\ \min\{g_{I(G_1)}(2),\,g_{I(G_2)}(1)-6,g_{I(G_3)}(1)-6,\,g_{I(G_3)}(2)-5\}.
	\end{align*}
	    
	Next, by using \textit{Macaulay2} \cite{GDS, FPack2} %the \textit{Macaulay2} \cite{GDS} package \texttt{MatchingPowers} \cite{FPack2} we checked that
	$g_{I(G_1)}(2)=3$, $g_{I(G_2)}(2)=7$, $g_{I(G_3)}(1)=8$ and $g_{I(G_3)}(2)=7$. Thus,
	$$
	g_{I(G)}(2)\ =\ \min\{3,\,1,\,2,\,2\}\ =\ 1.\\
	$$
\end{example}

Now we are ready to prove Theorem \ref{Thm:gITreeNonInc}.

\begin{proof}[Proof of Theorem \ref{Thm:gITreeNonInc}]
	Let $G_1,\dots,G_c$ be the connected components of $G$. Each $G_i$ is a tree with at least two vertices. Let $m=\max_{i}|V(G_i)|$. If $m=2$, then up to a relabeling, $I(G)=(x_1x_2,x_3x_4,\dots,x_{n-1}x_n)$ is a complete intersection. Setting $y_i=x_{2i-1}x_{2i}$, $i=1,\dots,\frac{n}{2}$, and $L=(y_1,\dots,y_{\frac{n}{2}})$, we have $\nu(G)=\nu(L)=\frac{n}{2}$ and $\pd(I(G)^{[k]})=\pd(L^{[k]})$ for all $k$. Since $L^{[k]}$ is a squarefree Veronese ideal generated in degree $k$, then $\pd(L^{[k]})=\frac{n}{2}-k$ for $k=1,\dots,\nu(L)$. Thus, for all $k=1,\dots,\frac{n}{2}$, from the Auslander--Buchsbaum formula, we have
	\begin{align*}
		g_{I(G)}(k)=\depth(S/I(G)^{[k]})-(2k-1)=n-1-\pd(I^{[k]})-(2k-1)=\frac{n}{2}-k.
	\end{align*}
	Hence, $g_{I(G)}(k)$ is non-increasing in this a case.
	
	Now suppose $m>2$. Then $G$ has a distant leaf. We proceed by induction on $n=|V(G)|$. If $\nu(G)=1$ there is nothing to prove. Similarly if $\nu(G)=2$, then $g_{I(G)}(2)=0$ \cite[Corollary 3.5]{EHHM2022b} and $g_{I(G)}(1)\ge0$. Thus we may suppose that $\nu(G)\ge3$. Assume Setup \ref{SetupG}. Let $k<\nu(G)$. Suppose $t=0$. Then, Corollary \ref{cor:gIRecursiveFormula} gives
	\begin{align}
		\label{eq:gIrec1}g_{I(G)}(k+1)&=\min\{g_{I(G_1)}(k+1),g_{I(G_2)}(k)-2,g_{I(G_3)}(k)-3,g_{I(G_3)}(k+1)-2\},\\
		\label{eq:gIrec2}g_{I(G)}(k)&=\min\{g_{I(G_1)}(k),g_{I(G_2)}(k-1)-2,g_{I(G_3)}(k-1)-3,g_{I(G_3)}(k)-2\}.
	\end{align}
	It remains to prove that $g_{I(G)}(k+1)-g_{I(G)}(k)\le0$. If $k+1=\nu(G)$, then $g_{I(G)}(k+1)=0$ by \cite[Corollary 3.5]{EHHM2022b} and there is nothing to prove.
	
	Suppose now $k+1<\nu(G)$. Then, we have $k\le\nu(G)-2$. It is easily seen that $\nu(G_3)\le\nu(G_2)\le\nu(G_1)\le\nu(G)$ and $\nu(G)=\nu(G_3)+1$. Thus $k\le\nu(G_3)-1$, and this inequality guarantees that all terms $g_{I(G_i)}(\ell)$ appearing in the minimum taken in (\ref{eq:gIrec1}) and in the minimum taken in (\ref{eq:gIrec2}) have $\ell\le\nu(G_i)$ for $i=1,2,3$.
	
	Now, to end the proof, it is enough to distinguish the four cases arising from (\ref{eq:gIrec2}). Suppose $g_{I(G)}(k)=g_{I(G_1)}(k)$. Then, from (\ref{eq:gIrec1}),
	we have $g_{I(G)}(k+1)\le g_{I(G_1)}(k+1)$ and
	\begin{align*}
		g_{I(G)}(k+1)-g_{I(G)}(k)\ &\le\ g_{I(G_1)}(k+1)-g_{I(G_1)}(k)\le0,
	\end{align*}
	where the last inequality follows by the inductive hypothesis, since $|V(G_1)|<n$. For the other three cases, one can argue similarly. The case $t>0$ is analogous.
\end{proof}

Lemma \ref{Lem:I(G)ForestBettiSplit} together with Lemma \ref{Lemma:BettiSplittingJ=J1+J2} and a simple inductive argument imply the following interesting consequence.

\begin{corollary}
	Let $G$ be a forest. Then, the graded Betti numbers of the squarefree powers $I(G)^{[k]}$ do not depend upon the characteristic of the field $K$.
\end{corollary}

We end the section with a general upper bound for the normalized depth function of the edge ideal of any graph $G$ in terms of the longest induced path of $G$.

We denote by $P_n$ the \textit{path on $n$ vertices}, that is the graph with $V(P_n)=[n]$ and $E(P_n)=\{\{1,2\},\{2,3\},\dots,\{n-1,n\}\}$. It is well known that $\nu(P_n)=\lfloor\frac{n}{2}\rfloor$.

\begin{theorem}\label{thm:path}
	We have
	$$
	g_{I(P_n)}(k)\ =\ \begin{cases}
		\lceil\frac{n}{3}\rceil-k&\text{if}\ k=1,\dots,\lceil\frac{n}{3}\rceil,\\
		\hfil0& otherwise.
	\end{cases}
	$$
\end{theorem}
\begin{proof}
	We proceed by induction on $n\ge3$. For $n=3,4,5$ one can easily verify the above formula. Let $n\ge6$, then $\nu(G)\ge3$ and %For the base case, we have $\nu(P_3)=1$ and \cite[Corollary 3.5]{EHHM2022b} gives $g_{I(P_3)}(1)=0=\lceil\frac{3}{3}\rceil-1$, as desired. Suppose $n>3$, then
	$n$ is a distant leaf with distant edge $\{n-1,n\}$.  Using the notation in Setup \ref{SetupG}, we have $t=0$, $G_1=P_{n-1}$, $G_2=P_{n-2}$ and $G_3=P_{n-3}$. Hence, by Corollary \ref{cor:gIRecursiveFormula},
	$$
	g_{I(P_n)}(k)=\min\{g_{I(P_{n-1})}(k),g_{I(P_{n-2})}(k-1)\!-\!2,g_{I(P_{n-3})}(k-1)\!-\!3,g_{I(P_{n-3})}(k)\!-\!2\}.
	$$
	
	If $k\ge\lceil\frac{n}{3}\rceil$, then $k-1\ge\lceil\frac{n-3}{3}\rceil$. By the inductive hypothesis $g_{I(P_{n-3})}(k-1)-3=0$, and by the above equation $g_{I(P_n)}(k)=0$, too. In this case, the assertion follows.
	
	Suppose now $1\le k<\lceil\frac{n}{3}\rceil$, then all terms appearing in the minimum taken above are greater than zero, by the inductive hypothesis. Hence,
	\begin{align*}
		g_{I(P_n)}(k)&=\min\Big\{\Big\lceil\frac{n-1}{3}\Big\rceil-k+1,\Big\lceil\frac{n-2}{3}\Big\rceil-(k-1),\Big\lceil\frac{n-3}{3}\Big\rceil-(k-1),\\
		&\phantom{=\min\Big\{.}\Big\lceil\frac{n-3}{3}\Big\rceil-k+1\Big\}\\
		&=\min\Big\{\Big\lceil\frac{n-1}{3}\Big\rceil-k+1,\Big\lceil\frac{n-2}{3}\Big\rceil-k+1,\Big\lceil\frac{n}{3}\Big\rceil-k\Big\}\\
		&=\Big\lceil\frac{n}{3}\Big\rceil-k,
	\end{align*}
	as desired. The inductive proof is complete.
\end{proof}

Denote by $\ell(G)$ the number of vertices of the longest induced path of $G$.

\begin{corollary}
	Let $G$ be a connected graph with $n$ vertices. Then $\nu(G)\ge\lfloor\frac{\ell(G)}{2}\rfloor$ and
	$$
	g_{I(G)}(k)\le\begin{cases}
		\big\lceil\frac{3n-2\ell(G)}{3}\big\rceil-k&\text{for}\ k=1,\dots,\big\lceil\frac{\ell(G)}{3}\big\rceil,\\
		\hfil n-\ell(G)&\text{for}\ k=\big\lceil\frac{\ell(G)}{3}\big\rceil+1,\dots,\big\lfloor\frac{\ell(G)}{2}\big\rfloor.
	\end{cases}
	$$
\end{corollary}
\begin{proof}
	Let $P$ be an induced path of $G$ with $|V(P)|=\ell(G)$, and set $S_P=K[x_v:v\in V(P)]$. Then $\nu(G)\ge\nu(P)=\lfloor\frac{\ell(G)}{2}\rfloor$. Let $k\le\lfloor\frac{\ell(G)}{2}\rfloor$. By \cite[Corollary 1.3]{EHHM2022a}, we have $\depth(S/I(G)^{[k]})\le\depth(S/I(P)^{[k]})=\depth(S_P/I(P)^{[k]})+n-\ell(G)$. Hence, the assertion follows from Theorem \ref{thm:path}.
\end{proof}

Our experiments using \textit{Macaulay2} \cite{GDS} suggest that $g_{I(P_n)}(k)=g_{I(C_n)}(k)$,  for all $2\le k\le\lfloor\frac{n}{2}\rfloor$, where $C_n$ is the cycle on $n$ vertices.

\begin{question}\label{que:cycle}
	Let $C_n$ denote the cycle on $n$ vertices. That is, $V(C_n)=[n]$ and $E(C_n)=E(P_n)\cup\{\{1,n\}\}$. Then $\nu(C_n)=\nu(P_n)$. It is well known that $\depth(S/I(P_n))=\lceil\frac{n}{3}\rceil$ and $\depth(S/I(C_n))=\lceil\frac{n-1}{3}\rceil$, thus $g_{I(P_n)}(1)$ and $g_{I(C_n)}(1)$ differ at most by one. Is it true that $$g_{I(C_n)}(k)=g_{I(P_n)}(k)$$ for all $2\le k\le\lfloor\frac{n}{2}\rfloor$?
\end{question}

\section{The Erey--Hibi Conjecture}
In this last section, we turn to the regularity of the squarefree powers of $I(G)$, when $G$ is a forest, and we prove a conjecture due to Erey and Hibi \cite{EH2021}.\smallskip

In \cite{EH2021}, Erey and Hibi introduced the concept of \textit{$k$-admissible matching} of a graph. Let $n$ and $k$ be positive integers. A sequence $(a_1,\dots,a_n)$ of positive integers is called a \textit{$k$-admissible sequence} if $a_1+\dots+a_n\le n+k-1$.

\begin{definition}\label{Def:kAdmMat}
	Let $G$ be a graph, $M$ be a matching of $G$ and $k\in[\nu(G)]$. We say that $M$ is a \textit{$k$-admissible matching} if there exists a sequence $M_1,\dots,M_r$ of non--empty subsets of $M$ satisfying the following conditions:
	\begin{enumerate}
		\item[(a)] $M=M_1\cup\dots\cup M_r$;
		\item[(b)] $M_i\cap M_j=\emptyset$ for all $i\ne j$;
		\item[(c)] for all $i\ne j$, if $e_i\in M_i$ and $e_j\in M_j$, then $\{e_i,e_j\}$ is a \textit{gap} in $G$, that is, $\{e_i,e_j\}$ is an induced matching of size $2$;
		\item[(d)] the sequence $(|M_1|,\dots,|M_r|)$ is $k$-admissible;
		\item[(e)] the induced subgraph of $G$ on $V(M_i)$ is a forest, for all $i\in[r]$.
	\end{enumerate}
	
	In such a case, we say that $M=M_1\cup\dots\cup M_r$ is a \textit{$k$-admissible partition} of $G$. The \textit{$k$-admissible matching number} of $G$, denoted by $\aim(G,k)$, is the number
	$$
	\aim(G,k)=\max\{|M|:M\ \textit{is a}\ k\textit{-admissible matching of}\ G\}
	$$
	for $k\in[\nu(G)]$. We set $\aim(G,k)=0$, if $G$ has no $k$-admissible matching.
\end{definition}

If $G$ is a forest, condition (e) is vacuously verified.
\begin{remark}\label{Rem:aim}(\cite[Remark 14]{EH2021})
	Note that $\aim(G,k)\ge k$ for all $k\in[\nu(G)]$. If $H$ is an induced subgraph of $G$, then $\aim(H,k)\le\aim(G,k)$. Moreover,
	$$
	\indm(G)=\aim(G,1)\le\aim(G,2)\le\dots\le\aim(G,\nu(G))=\nu(G).
	$$
\end{remark}

For the proof of the next corollary, we recall some facts. Let $u\in S$ be a monomial and $L\subset S$ be a monomial ideal, then $\reg(uL)=\reg(L)+\deg(u)$. Moreover, if $\{x_{p_1},\dots,x_{p_s}\}$ is a collection of variables not dividing any generator of $L$, then \cite[Corollary 3.2]{HRR} $\reg((x_{p_1},\dots,x_{p_s})L)=\reg(L)+1$. Furthermore, if $H$ is an induced subgraph of $G$, \cite[Corollary 1.3]{EHHM2022a} gives
$$
\reg(I(H)^{[k]})\le\reg(I(G)^{[k]}).
$$

\begin{corollary}\label{Cor:regRecursForest}
	With the notation and assumptions of Lemma \ref{Lem:I(G)ForestBettiSplit}, we have
	$$
	\reg(I(G)^{[k]})=\max\{\reg(I(G_1)^{[k]}),\reg(I(G_2)^{[k-1]})+2,\reg(I(G_3)^{[k]})+1\}
	$$
	for $k=1, \ldots, \nu(G)$.
\end{corollary}
\begin{proof}
	By Lemma \ref{Lem:I(G)ForestBettiSplit}, since $I(G)^{[k]}$ is a Betti splitting, from \cite[Corollary 2.2]{FHT2009} we have
	\begin{equation}\label{eq:reg(I(G)^{[k]})}
		\begin{aligned}
			\reg(I(G)^{[k]})&=\max\{\reg(I(G_1)^{[k]}),\reg(x_nx_{n-1}I(G_2)^{[k-1]}),\reg(J)-1\}\\
			&=\max\{\reg(I(G_1)^{[k]}),\reg(I(G_2)^{[k-1]})+2,\reg(J)-1\},
		\end{aligned}
	\end{equation}
	where $J$ is given in (\ref{eq:intersectBettiSplitGForest}). Since $\reg((x_{i_1},\dots,x_{i_t})J_1)-1=\reg(J_1)$, Lemma \ref{Lemma:BettiSplittingJ=J1+J2} gives
	\begin{align*}
		\reg(J)&=\max\{\reg(J_1),\reg(J_2),\reg((x_{i_1},\dots,x_{i_t})J_1)-1\}\\
		&=\max\{\reg(J_1),\reg(J_2)\},
	\end{align*}
	
	Since $J_1=x_{n}x_{n-1}[I(G_3)^{[k]}+x_{n-2}I(G_3)^{[k-1]}]$ is a Betti splitting, by Lemma \ref{Lemma:BettiSplit(I,x_n)} we have
	\begin{equation}\label{eq:reg(J_1)}
		\begin{aligned}
			\reg(J_1)&=\max\{\reg(x_{n}x_{n-1}I(G_3)^{[k]}),\reg(x_nx_{n-1}x_{n-2}I(G_3)^{[k-1]}),\\
			&\phantom{=\max\{}\reg(x_nx_{n-1}x_{n-2}I(G_3)^{[k]})-1\}\\
			&=\max\{\reg(I(G_3)^{[k]})+2,\reg(I(G_3)^{[k-1]})+3\}.
		\end{aligned}
	\end{equation}
	
	Now we distinguish two cases. \medskip\\
	\textsc{Case 1.} Suppose $t>0$, then $J_2=x_nx_{n-1}(x_{i_1},\dots,x_{i_t})I(G_2)^{[k-1]}\ne(0)$, and we have $\reg(J_2)=\reg(I(G_2)^{[k-1]})+3$. Since $\reg(I(G_3)^{[k-1]})\le\reg(I(G_2)^{[k-1]})$,
	\begin{equation}\label{eq:reg(J)1}
		\begin{aligned}
			\reg(J)&=\max\{\reg(I(G_3)^{[k]})+2,\reg(I(G_3)^{[k-1]})+3,\reg(I(G_2)^{[k-1]})+3\}\\
			&=\max\{\reg(I(G_3)^{[k]})+2,\reg(I(G_2)^{[k-1]})+3\}.
		\end{aligned}
	\end{equation}
	Hence, combining (\ref{eq:reg(I(G)^{[k]})}) with (\ref{eq:reg(J)1}), the assertion follows.\smallskip\\
	\textsc{Case 2.} Suppose $t=0$, then $J_2=(0)$ and $J=J_1$. Combining (\ref{eq:reg(J_1)}) with (\ref{eq:reg(I(G)^{[k]})}), the assertion follows because $\reg(I(G_3)^{[k-1]})\le\reg(I(G_2)^{[k-1]})$.
\end{proof}

We need the following lemmata proved in \cite{EH2021}. %\cite[Lemma 15 and Lemma 23]{EH2021}.
\begin{lemma}{\rm \cite[Lemma 15]{EH2021}}\label{Lemma:aimEreyHibiStessoG}
	Let $G$ be a forest and let $2\le k\le\nu(G)$. Then
	$$
	\aim(G,k)\le\aim(G,k-1)+1.
	$$
\end{lemma}

\begin{lemma}{\rm \cite[Lemma 23]{EH2021}}\label{Lemma:aimEreyHibi}
	Let $G$ be a forest and let $\{w,v\}$ be a distant edge of $G$. Then, for all $2\le k\le\nu(G)$,
	$$
	\aim(G\setminus\{w,v\},k-1)+1\le\aim(G,k).
	$$
\end{lemma}

We are in the position to prove Erey and Hibi conjecture \cite[Conjecture 31]{EH2021}.
\begin{theorem}\label{Thm:ConjEreyHibi}
	Let $G$ be a forest. Then
	$$
	\reg(I(G)^{[k]})=\aim(G,k)+k, \ \ \ \text{for}\ k=1,\dots,\nu(G).
	$$
\end{theorem}
\begin{proof}
	As in the proof of Theorem \ref{Thm:gITreeNonInc}, let $G_1,\dots,G_c$ be the connected components of $G$, and let $m=\max_{i}|V(G_i)|$.
	
	If $m=2$, up to a relabeling, $I(G)=(x_1x_2,x_3x_4,\dots,x_{n-1}x_n)$. Since $M=\{\{1,2\},\{3,4\},\dots,\{n-1,n\}\}$ is a $k$-admissible matching for all $k=1,\dots,\nu(G)$, we have $\aim(G,k)=\frac{n}{2}$ for all $k=1,\dots,\nu(G)$, and $\indm(G)=\nu(G)=\frac{n}{2}$. In this case, the equality $\reg(I(G)^{[k]})=\aim(G,k)+k$ follows from \cite[Proposition 30]{EH2021}.
	
	Suppose $m>2$. We proceed by induction on $|V(G)|\ge3$. We assume $\nu(G)\ge3$. Indeed, $\reg(I(G)^{[1]})=\aim(G,1)+1$ and $\reg(I(G)^{[\nu(G)]})=\aim(G,\nu(G))+\nu(G)=2\nu(G)$ have already been proved in \cite[Theorem 4.7]{BHT2015} and  \cite[Theorem 5.1]{BHZN18} (see also \cite{EH2021}). So, if $\nu(G)\le 2$, there is nothing to prove. Thus, we may assume Setup \ref{SetupG}. Then, Corollary \ref{Cor:regRecursForest} together with the inductive hypothesis yield
	\begin{align*}
		\reg(I(G)^{[k]})&=\max\{\reg(I(G_1)^{[k]}),\reg(I(G_2)^{[k-1]})+2,\reg(I(G_3)^{[k]})+1\}\\
		&=\max\{\aim(G_1,k)+k,\aim(G_2,k-1)+k+1,\aim(G_3,k)+k+1\}.
	\end{align*}
	
	To prove the assertion, it is enough to show that
	\begin{equation}\label{eq:aimRecursive}
		\aim(G,k)=\max\{\aim(G_1,k),\aim(G_2,k-1)+1,\aim(G_3,k)+1\}.
	\end{equation}
	Firstly, we prove that
	\begin{equation}\label{eq:aimge}
		\aim(G,k)\ge\max\{\aim(G_1,k),\aim(G_2,k-1)+1,\aim(G_3,k)+1\}.
	\end{equation}
	Indeed, $\aim(G,k)\ge\aim(G_1,k)$ by Remark \ref{Rem:aim}. Moreover, by Lemma \ref{Lemma:aimEreyHibi} it follows that $\aim(G,k)\ge\aim(G_2,k-1)+1$, because $G_2=G\setminus\{n-1,n\}$ and $\{n-1,n\}$ is a distant edge. Finally, we prove $\aim(G,k)\ge\aim(G_3,k)+1$. Indeed, let $M$ be a $k$-admissible matching of $G_3$ with $|M|=\aim(G_3,k)$. Then $M=M_1\cup\dots\cup M_r$ with $|M_1|+\dots+|M_r|\le r+k-1$ as in Definition \ref{Def:kAdmMat}. Note that $\{n-1,n\}$ forms a gap with all $e\in E(G_3)$. Thus, setting $M'=M\cup\{\{n-1,n\}\}$ we have that $M'$ is a $k$-admissible matching of $G$, because $M'$ admits the $k$-admissible partition $M_1\cup\dots\cup M_r\cup\{\{n-1,n\}\}$ since
	$$
	|M_1|+\dots+|M_r|+|\{\{n-1,n\}\}|\le (r+1)+k-1.
	$$
	This shows that $\aim(G,k)\ge\aim(G_3,k)+1$.
	
	Now we show that (\ref{eq:aimRecursive}) holds. If $\aim(G,k)=\aim(G_2,k-1)+1$, then equality follows by (\ref{eq:aimge}). Suppose now that $\aim(G,k)>\aim(G_2,k-1)+1$. By Lemma \ref{Lemma:aimEreyHibiStessoG}, $\aim(G_2,k-1)+1\ge\aim(G_2,k)$. Hence $\aim(G,k)>\aim(G_2,k)$. So we can find a $k$-admissible matching $M$ of $G$ with $|M|=\aim(G,k)>\aim(G_2,k)$ such that $M=M_1\cup\dots\cup M_r$ is a $k$-admissible partition. Thus $n-1\in V(M)$, otherwise $M$ is also a $k$-admissible matching of $G_2$, which is impossible. One of the following possibilities occurs: $\{n-1,n-2\}\in M$, $\{n-1,n\}\in M$ or $\{n-1,i_j\}\in M$, $j\in[t]$.
	
	Suppose $\{n-1,n-2\}\in M$, up to relabeling we may assume $\{n-1,n-2\}\in M_1$. Set $M'=M\setminus\{\{n-1,n-2\}\}$. We claim that $M_1$ is a singleton. If not, then
	$$
	|M_1\setminus\{\{n-1,n-2\}\}|+|M_2|+\dots+|M_r|=|M|-1\le r+(k-1)-1.
	$$
	By the definition of admissible $(k-1)$-matching, the above inequality implies that $\aim(G_3,k-1)\ge|M|-1$. Hence, $\aim(G_2,k-1)+1\ge\aim(G_3,k-1)+1\ge |M|=\aim(G,k)$, against our assumption.
	
	Hence, $M_1=\{\{n-1,n-2\}\}$ and $M'$ is a $k$-admissible matching of $G_3$, because $|M_2|+\dots+|M_r|\le (r-1)+k-1$. Thus $\aim(G_3,k)+1\ge\aim(G,k)$. By (\ref{eq:aimge}), $\aim(G,k)\ge\aim(G_3,k)+1$. Hence $\aim(G,k)=\aim(G_3,k)+1$ and (\ref{eq:aimRecursive}) holds.
	
	The cases $\{n-1,n\}\in M$, $\{n-1,i_j\}\in M$, $j\in[t]$, can be treated similarly. Indeed, $N_G(n-1)=\{i_1,\dots,i_t,n-2,n\}$ with $t\ge0$, as stated in Setup \ref{SetupG}.
\end{proof}

\end{document}